\documentclass{article}
\usepackage[utf8]{inputenc}
\usepackage[a4paper, total={6in, 8in}]{geometry}
\usepackage{mathrsfs}
\usepackage{amssymb}
\usepackage{amsmath}
\usepackage{setspace}
\usepackage{tikz}
\usetikzlibrary{positioning}
\usepackage{enumerate}
\usepackage{hyperref}
\usepackage{bm}
\usepackage{array}
\usepackage{comment}
\usepackage{amsthm}
\usepackage{enumerate}
\numberwithin{equation}{section}
\newcolumntype{C}{>{$}c<{$}}
\hypersetup{
    colorlinks=true,      
    urlcolor=cyan,
    citecolor=blue
    }
\newcommand{\parens}[1]{\left( #1 \right)}
\newcommand{\set}[1]{\left\{ #1 \right\}}
\newcommand{\tensor}{\otimes}
\newcommand{\spn}{\operatorname{span}}
\newcommand{\bbar}[1]{\overline{#1}}
\newcommand{\gl}{\mathfrak{gl}}
\renewcommand{\sp}{\mathfrak{sp}}

\renewcommand{\arraystretch}{1}
\newcommand{\opad}{\operatorname{ad}}
\newcommand{\dprod}[2]{\bbar #1 \diamond \bbar #2}

\theoremstyle{plain}
\newtheorem{lemma}{Lemma}[section]
\newtheorem{proposition}[lemma]{Proposition}
\newtheorem{theorem}[lemma]{Theorem}

\theoremstyle{definition}

\newtheorem{example}[lemma]{Example}
\newtheorem{remark}[lemma]{Remark}

\title{Reduction Algebras Involving Symplectic Lie Algebras}
\author{Dustin Baker}
\date{\today}

\begin{document}

\maketitle

\begin{abstract}


A complete presentation for the diagonal reduction algebra of $\mathfrak{gl}_n$ was found by Khoroshkin and Ogievetsky. In this paper we present a complete set of ordering relations for a reduction algebra $\mathcal{B}_n$ associated to $\gl_n$ in a parabolic subalgebra of $\sp_{2n}$. We also study its connection to the diagonal reduction algebra of $\sp_{2n}$.  We use braid group action by Zhelobenko automorphisms to reduce the computations required. Another important tool we employ is the ABRR equation for the extremal projector. As an application we prove that all homogeneous ordering relations for the diagonal reduction algebra of $\sp_{2n}$ are obtained from the relations in $\mathcal{B}_n$ by applying type $C_n$ Zhelobenko automorphisms.

\end{abstract}
\section{Introduction} \label{sec:intro}

A classic problem in representation theory is understanding how irreducible representations $V$ of a finite-dimensional complex Lie algebra $\mathfrak{g}$ decompose into a direct sum of irreducible representations of a semisimple subalgebra $\mathfrak{k}$. Given a triangular decomposition $\mathfrak{k}=\mathfrak{n}_-\oplus\mathfrak{h}\oplus\mathfrak{n}_+$,  \cite{Mick1973} introduced the \emph{step algebra} $\mathcal{S}(\mathfrak{g},\mathfrak{k})$, which we now call the \emph{Mickelsson algebra}, which acts on the space of $\mathfrak{n_+}$-invariant vectors in $V$. In fact, $V^{\mathfrak{n_+}}$ is an irreducible $S(\mathfrak{g},\mathfrak{k})$-module \cite{Homb1975}. The Mickelsson algebra describes how irreducible representations $V$ of $\mathfrak{g}$ decompose as a direct sum of irreducible representations of $\mathfrak{k}$.
In particular, $S(\mathfrak{g}, \mathfrak{k})$ contains step operators which map the highest weight vector of one subrepresentation to that of another.

The complicated structure of Mickelsson algebras has since become a subject of interest independent of its origin. In the late 1980's, Zhelobenko observed a connection between extremal projectors and localized Mickelsson algebras, denoted $\mathcal{Z}(\mathfrak{g},\mathfrak{k})$ \cite{Zhe1987}. Following \cite{Kho2008}, we call these \emph{reduction algebras}.
Those associated to $\mathfrak{g}$ diagonally embedded into $\mathfrak{g}\times\mathfrak{g}$ are called \emph{diagonal} reduction algebras. Reduction algebras are defined for a finite dimensional Lie algebras $\mathfrak{g}$ and Lie subalgebra $\mathfrak{k}$ that are reductive in $\mathfrak{g}$. There are also generalizations to Lie superalgebras \cite{Hart2022} and quantum groups \cite{Mud2015}.

In the 2010's, Khoroshkin and Ogievetsky provided a complete presentation of a certain class of reduction algebras called a diagonal reduction algebra, namely of $\gl_n$ in $\gl_n\times \gl_n$. They made use of Zhelobenko automorphisms to reduce the complexity
of the problem \cite{Kho2011}. These form a representation of the braid group by algebra automorphisms of the reduction algebra.

When $\mathfrak{k}$ is reductive in $\mathfrak{g}$,  we have $\mathfrak{g} = \mathfrak{k}\oplus V$, where $V$ is complimentary $\mathfrak{k}$-module. 
The presentation of a reduction algebra can be given via \emph{ordering relations} \cite{Zhe1989}, relations between elements of a particular ordered generating set $(x_1,\ldots, x_n)$ of $\mathfrak{g}$ of the form 
\begin{equation}
\label{eq: ordering relation form}
x_jx_i = \sum_{k < \ell} h_{ijk\ell}x_{k}x_{\ell} + \sum_{k} h_{ijk} x_k + h_{ij},
\end{equation} 
where $h_{ijk\ell},h_{ijk},h_{ij}$ belong to a localization of $U(\mathfrak{h})$  (see Theorem \ref{thm: generators and ordering relations} for details).

In this paper, we provide a complete presentation for the reduction algebra  $\mathcal{B}_{n} = \mathcal{Z}(\mathfrak{p}_{2n},\gl_n)$ associated to $\gl_n$ in $\mathfrak{p}_{2n}$ where $\mathfrak{p}_{2n}$ is the maximal parabolic subalgebra of $\sp_{2n}$ obtained by removing the negative long simple root (Theorem \ref{thm: All relations for Bn}). We compute only seven ordering relations directly, and obtain the rest by the action of the braid group of type $A_n$. This requires careful attention, as the braid group action must preserve the order of every term in a given ordering relation. Details can be found in Theorem \ref{thm: All relations for Bn}.

In addition, we explore alternate realizations of $\mathcal{B}_n$ (Theorem \ref{thm: alt realizations}). We show $\mathcal{B}_n$ fits into the family of reduction algebras $\mathcal{Z}(\gl_n\ltimes S^m(\mathbb{C}^n), \gl_n)$ with $m=2$; The case when $m=1$ was addressed by \cite{Homb1976}. We also prove that $\mathcal{B}_n$ arises as a subalgebra the diagonal reduction algebra $\mathcal{D}(\sp_{2n}) = \mathcal{Z}(\sp_{2n}\times\sp_{2n}, \sp_{2n})$.

Expanding on Theorem \ref{thm: All relations for Bn}, we show that every \emph{homogeneous} ordering relation (those of the form in Equation \eqref{eq: ordering relation form} with $h_{ijk}=h_{ij}=0$) for $\mathcal{D}(\sp_{2n})$ can be obtained from $\mathcal{B}_n$ by the action of the braid group of type $C_n$ (Theorem \ref{thm: ord rels for non roots}).

\section*{Acknowledgement}
The author was supported in part by the Army Research Office grant W911NF-24-1-0058.

\section{Global Definitions and Notation} \label{sec: defs}
We define the reduction algebra and describe the tools we use to understand its properties. Throughout this paper, all vector spaces and algebras are over $\mathbb{C}$.
\subsection{Prerequisites} \label{sec: prerequisites}
Let $\mathfrak{k}$ be a finite-dimensional reductive Lie algebra over $\mathbb{C}$, and let $\mathfrak{g}$ be a Lie algebra over $\mathbb{C}$ containing $\mathfrak{k}$ such that the adjoint action of $\mathfrak{k}$ on $\mathfrak{g}$ is completely reducible. Fix a Cartan subalgebra $\mathfrak{h}$ of $\mathfrak{k}$ with associated root system $\Delta$ of $\mathfrak{k}$. Let $\Delta_+$ be a choice of positive roots, and $\Delta_-$ the corresponding negative roots. Let $\Pi$ be the set of simple roots.
For each positive root $\beta\in\Delta_+$ we fix an $\mathfrak{sl}_2$-triple $(e_{-\beta},\, h_\beta,\, e_\beta)$.
We write \[\mathfrak{k} = \mathfrak{n}_-\oplus\mathfrak{h}\oplus\mathfrak{n}_+,\] where $\mathfrak{n}_\pm=\spn_\mathbb{C}\set{e_\beta\mid \beta\in \Delta_\pm}$. Denote the universal enveloping algebra of $\mathfrak{g}$ by $U(\mathfrak{g})$. Let $I_+ := U(\mathfrak{g})\mathfrak{n}_+$ and $I_-:= \mathfrak{n}_-U(\mathfrak{g})$; these are left and right ideals, respectively. The \textit{Mickelsson} algebra (or \textit{step algebra}) $\mathcal{S}(\mathfrak{g},\mathfrak{k})$ introduced in \cite{Mick1973} is given by 
\begin{equation}\label{eq:mickelsson-alg}
\mathcal{S}(\mathfrak{g}, \mathfrak{k}) := N(I_+)/I_+ = \set{u+I_+\in U(\mathfrak{g})/I_+\mid I_+u\subseteq I_+},
\end{equation}
where $N(I_+)$ is the normalizer of $I_+$, the largest subalgebra containing $I_+$ as a two-sided ideal.

Following \cite{Zhe1987} we consider the reduction algebra $\mathcal{Z}(\mathfrak{g}, \mathfrak{k})$, defined by \begin{equation}\mathcal{Z}(\mathfrak{g},\mathfrak{k}) :=N(I'_+)/I'_+, \end{equation} where $I'_+ = U'(\mathfrak{g})\mathfrak{n}_+$. Here $U'(\mathfrak{g}) = D \tensor_{U(\mathfrak{h})} U(\mathfrak{g})$ where $D = \mathbb{C}[h_\gamma, (h_\gamma + k)^{-1}\mid \gamma\in \Delta_+,k\in \mathbb{Z}]$. That is, $U'(\mathfrak{g})$ is the localization of $U(\mathfrak{g})$ with respect to the Ore denominator set generated by $\set{h_\gamma+k\mid \gamma\in \Delta_+,k\in\mathbb{Z}}$. 
Lastly, the localization of a step algebra is isomorphic to Zhelobenko's reduction algebra:
\begin{equation}\label{eq:localization-commutes}
    D\otimes_{U(\mathfrak{h})} S(\mathfrak{g},\mathfrak{k}) \cong \mathcal{Z}(\mathfrak{g},\mathfrak{k}).
\end{equation}

\subsection{Extremal Projector} \label{sec: extremal proj general}

The extremal projector  \cite{Tol2005},\cite{Zhe1989} plays an important role in understanding the structure of $Z(\mathfrak{g},\mathfrak{k})$. It is an element of the Taylor extension of $U'(\mathfrak{g})$, denoted $T(U'(\mathfrak{g}))$ \cite{Herl2017}. It is uniquely determined by the conditions \cite{Tol2005} \begin{equation}[\mathfrak{h},P]=0,\quad e_\gamma P = P e_{-\gamma} = 0 \quad (\gamma\in \Delta_+), \qquad P^2=P, P\neq 0.\end{equation} We describe explicitly the multiplicative form of $P$. First, there exists a \textit{convex} (or \textit{normal}) order on $\Delta_+$; that is, a total order in which every positive root that can be expressed as the sum of two positive roots must lie between its summands. Let $\overset{\rightarrow}{\Delta}_+= (\gamma_1,\ldots, \gamma_m)$ be such an ordering. For each $\gamma\in \Delta_+$ and $t\in \mathbb{Z}$, define \begin{equation}P_\gamma[t] = \sum_{n\geq 0} \frac{(-1)^n}{n!} \varphi_{\gamma,n}[t] e_{-\gamma}^n e_\gamma^n,\end{equation} where $\varphi_{\gamma,n}[t] = \prod_{k=1}^n (h_\gamma + t + k)^{-1}$ (with empty product $1$).
\begin{proposition} \label{prop: extremal proj general}
\cite{Ast1973},\cite{Herl2017} The extremal projector $P$ admits a factorization in the ordered product \begin{equation} P = \underset{{\gamma\in \overset{\rightarrow}{\Delta}_+}}{\overset{\rightarrow}{\prod}} P_{\gamma}[\rho(h_{\gamma})]\end{equation} for any choice of normal order $\set{\gamma_1,\ldots, \gamma_m}$ on $\Delta_+$ and $\rho = \frac{1}{2}\sum_{i=1}^m \gamma_i$ is the Weyl vector.
\end{proposition} 
The Chevalley anti-involution is the anti-automorphism $\epsilon: \mathfrak{g}\to \mathfrak{g}$ for which $e_{\pm\gamma}\mapsto e_{\mp \gamma}$ for $\gamma\in \Delta_+$ and $h\mapsto h$ for $h\in \mathfrak{h}$. It can be extended to $U'(\mathfrak{g})$ and further to $T(U'(\mathfrak{g}))$. A noteworthy consequence of the factorization of $P$ is that it is fixed by the Chevalley anti-involution.

\subsection{Double Coset Algebra} \label{sec: double coset alg general}
The double coset space $I_- \backslash U'(\mathfrak{g})/I_+$ is closely related to the reduction algebra. For convenience we write $\mathbb{I} = I_- + I_+$ and denote the double coset space by $U'(\mathfrak{g})/\mathbb{I}$. We can equip the double coset space with the \textit{diamond} product, given by \begin{equation}(a+\mathbb{I})\diamond(b+\mathbb{I}):=aPb + \mathbb{I}.\end{equation} This can be interpreted as follows \cite{Kho2008}: given $\bbar{a},\bbar{b}\in U'(\mathfrak{g})/\mathbb{I}$, take representatives $a\in U'(\mathfrak{g})$ and $b\in U'(\mathfrak{g})/I_+$. Here, $U'(\mathfrak{g})/I_+$ is the \textit{relative\footnote{We call it relative because $I_+$ is generated by $\mathfrak{k}_+$, not $\mathfrak{g}_+$.} universal Verma module}, and is notably a left $U'(\mathfrak{g})$-module. We can interpret $P$ as a map $P: U'(\mathfrak{g})/I_+\to U'(\mathfrak{g})/I_+$ as in \cite{Zhe1989}. The diamond product is computed by \begin{equation}(a+\mathbb{I})\diamond(b+\mathbb{I}) = a.P(b+I_+) + I_-;\end{equation} that is, apply the map $P$ to $b+I_+$, act by $a$, then mod out by $I_-$. We usually write $\bbar{a} = a +\mathbb{I}$ for simplicity. The definition of $\diamond$ is independent of the choice of representatives $a$ and $b$. \\
The significance of the double coset space is made clear by a remarkable result of \cite{Kho2008}, which goes back to \cite[Theorem 1]{Mick1973}:
\begin{proposition} \label{prop: double coset realization isomorphism}
The space $(U'(\mathfrak{g})/\mathbb{I},\diamond)$ is an associative algebra and is isomorphic to $\mathcal{Z}(\mathfrak{g}, \mathfrak{k})$ via mutually inverse isomorphisms \begin{equation} \phi: \mathcal{Z}(\mathfrak{g}, \mathfrak{k})\to U'(\mathfrak{g})/\mathbb{I}; \qquad (x+I_+)\mapsto (x + \mathbb{I}),\end{equation} and \begin{equation} P:U'(\mathfrak{g})/\mathbb{I}\to \mathcal{Z}(\mathfrak{g}, \mathfrak{k}); \qquad (x+\mathbb{I})\mapsto P(x+I_+). \end{equation}
\end{proposition}

Since $\mathfrak{k}$ is reductive in $\mathfrak{g}$, we have $\mathfrak{g} = \mathfrak{k}\oplus\mathfrak{p}$ for some complimentary $\mathfrak{k}$-module $\mathfrak{p}$. Let $\set{p_i}$ be a weight basis for $\mathfrak{p}$ such that $p_i$ has weight $\gamma_i$. We choose this basis to be compatible with a partial order on the weights: $\gamma_j - \gamma_i\in \Delta_+$ only if $i < j$. Combining Proposition \ref{prop: double coset realization isomorphism} with \cite[Theorem 6]{Zhe1989} gives the following theorem:

\begin{theorem} \label{thm: generators and ordering relations}
The double coset space $U'(\mathfrak{g})/\mathbb{I}$ is generated as a $U'(\mathfrak{h})$-ring by $\set{\bbar{p}_i}$, along with the weight relations \begin{equation}
[h,\bbar{p}_i] = \gamma_i(h)\bbar{p}_i
\end{equation} and ordering relations \begin{equation}
\dprod{p_j}{p_i} = \sum_{k,l; k\leq l} a_{ijkl} \dprod{p_k}{p_l} + \sum_k b_{ijk} \bbar{p}_k + c_{ij},
\end{equation} where $a_{ijkl},b_{ijk},c_{ij}\in U'(\mathfrak{h})$.
\end{theorem}

We now shift our attention to the double coset algebra. With the goal of finding ordering relations of $U'(\mathfrak{g})/\mathbb{I}$, we introduce a useful tool to reduce the complexity of this task.

\subsection{Zhelobenko Automorphisms} \label{sec: Zhelo autos general}
For each $\gamma\in \Delta_+$, the map $\tau_\gamma: \mathfrak{g}\to\mathfrak{g}$ given by \begin{equation}\tau_\gamma = (\exp \opad e_\gamma)\circ \exp(-\opad e_{-\gamma})\circ(\exp\opad e_\gamma)\end{equation} is a Lie algebra automorphism for which $\tau_\gamma(e_{\pm\gamma}) = -e_{\mp\gamma}$ and $\tau(h_\gamma) = -h_\gamma$. For simple roots $\gamma\in \Pi$ these give rise to automorphisms of $\mathcal{Z}(\mathfrak{g},\mathfrak{k})$, called \emph{Zhelobenko automorphisms} $\breve{q}_\gamma$ \cite{Kho2008}, given by \begin{equation} \breve{q}_\gamma(x) = \sum_{n\geq 0} \frac{1}{n!} g_{n,\gamma} (\opad e_\gamma)^n (\opad e_{-\gamma})^n(\tau_\gamma(x)) + \mathbb{I},\end{equation} where $g_{n,\gamma} = \prod_{k=1}^n (h_\gamma - k +1)^{-1}.$ If $\Pi = \set{\alpha_1,\ldots, \alpha_n}$, we write $\tau_i:=\tau_{\alpha_i}$ and $\breve{q}_i := \breve{q}_{\alpha_i}$. Both $\tau_\gamma$ and $\breve{q}_\gamma$ satisfy braid relations; that is, with $\sigma_\gamma = \tau_\gamma$ or $\breve{q}_\gamma$, 
\begin{equation} \sigma_i \sigma_j\sigma_i\cdots = \sigma_j\sigma_i\sigma_j\cdots\end{equation} where the number of terms on each side is $m_{ij}$ and is determined by the corresponding entry $a_{ij}$ of the Cartan matrix for $\mathfrak{k}$: $m_{ij} = 2$ if $a_{ij} = 0$, $m_{ij} = 3$ if $a_{ij}a_{ji} = 1$, $m_{ij} = 4$ if $a_{ij}a_{ji} = 2$, and $m_{ij}=6$ if $a_{ij}a_{ji}=3$.

\section{The Reduction Algebra \texorpdfstring{$\mathcal{Z}(\mathfrak{p}_{2n}, \gl_n)$}{Z(p2n,gln)}} \label{sec: red alg of max para}

\subsection{The Symplectic Lie Algebra} \label{sec: symplectic Lie alg}
The symplectic Lie algebra $\sp_{2n}$ (over $\mathbb{C}$) can be realized as \begin{equation}\sp_{2n} = \set{\begin{bmatrix}
A & B \\ C & -A^\top
\end{bmatrix}: A,B,C\in \mathbb{C}^{n\times n}, B=B^\top, C=C^\top}\end{equation} equipped with the commutator bracket. We choose a basis for $\sp_{2n}$ in terms of the standard basis $\set{E_{ij}}_{1\leq i,j\leq 2n}$for $\mathbb{C}^{2n\times 2n}$: \begin{align}
A_{ij} &= E_{ij} - E_{n+j,n+i} &&\text{for $1\leq i,j\leq n$,} \\
B_{ii} &= E_{i,n+i} && \text{for $1\leq i\leq n$,} \\
B_{ij} &= E_{i,n+j} + E_{j,n+i} && \text{for $1\leq i < j\leq n$,} \\
C_{ij} &= B_{ij}^\top && \text{for $1\leq i\leq j\leq n$.}
\end{align}
It is convenient to define $B_{ji} = B_{ij}$ and $C_{ji}=C_{ij}$ when $j > i$.
The Lie subalgebra $\spn_{\mathbb{C}}\set{A_{ij}\mid 1\leq i,j\leq n}$ is isomorphic to $\gl_{n}$. The Lie subalgebra \begin{equation}
\mathfrak{p}_{2n} = \spn_\mathbb{C}\set{A_{ij},B_{ij}\mid 1\leq i,j\leq n}    
\end{equation} is a maximal parabolic subalgebra of $\sp_{2n}$. In this way, $(\mathfrak{p}_{2n},\gl_n)$ is a reductive pair, hence we can define the reduction algebra $\mathcal{Z}(\mathfrak{p}_{2n},\gl_n)$ as in Section \ref{sec: prerequisites}.\\

\subsection{Extremal Projector} \label{sec: extremal proj Bn case}

We write $\gl_n = \mathfrak{n}_-\oplus \mathfrak{h}\oplus \mathfrak{n}_+$ where $\mathfrak{n}_{\pm} = \spn\set{A_{ij}\mid i \lessgtr j, 1\leq i,j\leq n}$ and $\mathfrak{h} = \spn\set{A_{ii}\mid 1\leq i\leq n}$. The positive root associated to $A_{ij}$ (with $i<j$) is $\varepsilon_i - \varepsilon_j$, where $\varepsilon_k(A_{\ell\ell}) = \delta_{k\ell}$. We often write $e_{\varepsilon_i- \varepsilon_j} = A_{ij}$.

The lexicographic order of the positive root vectors corresponds to a normal order of the positive roots; that is, $A_{ij} < A_{k\ell}$ when $i < k$ or $i=k$ and $j < \ell$. Indeed, the only way a positive root can be the sum of two other positive roots is \begin{equation*}
\varepsilon_i - \varepsilon_j = (\varepsilon_i - \varepsilon_k) + (\varepsilon_k - \varepsilon_j)
\end{equation*} with $i < k < j$, which coincides with the order $A_{ik} < A_{ij} < A_{kj}$. 
The Weyl vector in this case is \begin{equation}
\rho = \sum_{k=1}^n \parens{\frac{n+1}{2}-k}\varepsilon_k
\end{equation}
Following Proposition \ref{prop: extremal proj general} the extremal projector has the form \begin{equation} \label{eq: extremal proj for Bn}
P = \prod_{1\leq i < j\leq n}^\text{lex} P_{\varepsilon_i - \varepsilon_j}[j-i],
\end{equation} 
where \begin{equation}
P_\gamma[t] = \sum_{k\geq 0} \frac{(-1)^k}{k!}\varphi_{\gamma,k}[t]e_{-\gamma}^k e_\gamma^k
\end{equation}
and \begin{equation}
\varphi_{\gamma,k}[t] = \prod_{l=1}^k (h_\gamma + t + l)^{-1}.
\end{equation} We can then define the double coset algebra as in Section \ref{sec: double coset alg general}, which we call \begin{equation} \mathcal{B}_n := (U'(\mathfrak{p}_{2n}) / \mathbb{I}, \diamond). \end{equation}
\subsection{Zhelobenko Automorphisms} \label{sec: Zhelo autos Bn case}

We name the simple roots $\alpha_i = \varepsilon_i - \varepsilon_{i+1}$ for $1\leq i < n$. The subscript `$i$' is used in place of `$\alpha_i$'. \\
The action of the automorphisms $\tau_i$ on $\mathcal{X}=\set{A_{ij}, B_{ij}}$ can be described by the rule: 
\begin{equation*}
\text{$\tau_i$ interchanges indices $i$ and $i+1$. The sign changes for every instance of $i$ replaced by $i+1$.}\end{equation*}

We may now describe the Zhelobenko automorphisms. There are only four types of images that must be described. Recall from Section \ref{sec: Zhelo autos general} that \begin{equation}
\breve{q}(\bbar{x}) = \sum_{n\geq 0}\frac{1}{n!} g_{n,\gamma} (\opad e_\gamma)^n(\opad e_{-\gamma})^n(\tau_\gamma(x)) + \mathbb{I},
\end{equation} where \(
g_{n,\gamma} = \prod_{k=1}^n (h_\gamma - k + 1)^{-1}. \) \begin{lemma} \label{lem: zhel auto Bn} For each $i < n$ we consider the image of $\mathcal{X}_i = (\mathcal{X}\cup\set{h_i,I_i})\setminus\set{A_{ii},A_{i+1,i+1}}$, where $I_i = A_{ii} + A_{i+1,i+1}$.
\begin{enumerate}[{\rm 1)}]
    \item $\breve{q}_i(\bbar{x}) = \frac{h_i+1}{h_i}\tau_i(x) + \mathbb{I}$ when  $x\in\set{ A_{i+1,j},A_{ji}, B_{i+1,j}}_{j\neq i,i+1}$.
    \item $\breve{q}_i(\bbar{x}) = \frac{h_i+2}{h_i}\tau_i(x) + \mathbb{I}$ when  $x\in\set{h_i,B_{i,i+1}}$.
    \item $\breve{q}_i(\bbar{x}) = \frac{h_i + 1}{h_i-1}\tau_i(x) + \mathbb{I}$ when  
     $x\in\set{ A_{i+1,i},B_{i+1,i+1}}$.
    \item $\breve{q}_i(\bbar{x}) = \tau_i(x) + \mathbb{I}$ for every other $x\in\mathcal{X}_i$.
\end{enumerate}
\end{lemma}

\begin{proof}
Let $\gamma=\alpha_i$ be a simple root. In all cases the elements of $\mathcal{X}_i$ belong to $\gl_{n}$. As such, $(\opad e_{-\gamma})^k=0$ for $k\ge 3$ and we only the first three terms of $\breve{q}_i$ may be nonzero. Thus for $x\in \mathcal{X}_i$, we have \begin{equation}
\breve{q}_i(\bbar{x}) = \tau_\gamma(x) + \frac{1}{h_\gamma} [e_\gamma,[e_{-\gamma}, \tau_\gamma(x)]] + \frac{1}{2h_\gamma (h_\gamma-1)}[e_\gamma,[e_\gamma,[e_{-\gamma},[e_{-\gamma},\tau_\gamma(x)]]]] + \mathbb{I}.
\end{equation}
Each element of $\mathcal{X}_i$ belongs to one of the following four cases:
\begin{enumerate}[{\rm 1)}]
    \item \([e_{-\gamma},[e_{-\gamma},\tau_\gamma(x)]] = 0, \quad [h_\gamma, \tau_\gamma(x)] = \tau_\gamma(x),\quad [e_\gamma,\tau_\gamma(x)] = 0,\) \\
    \item \([e_{-\gamma},[e_{-\gamma},\tau_\gamma(x)]] = 0, \quad [h_\gamma, \tau_\gamma(x)] = 0,\quad [e_\gamma,[e_{-\gamma},\tau_\gamma(x)]] = 2\tau_\gamma(x),\) \\
    \item \([e_\gamma,[e_\gamma,[e_{-\gamma},[e_{-\gamma},\tau_\gamma(x)]]]] = 4\tau_\gamma(x), \quad [h_\gamma, \tau_\gamma(x)] = 2\tau_\gamma(x),\quad [e_\gamma,\tau_\gamma(x)] = 0,\) \\
    \item \([e_{-\gamma},\tau_\gamma(x)] = 0.\)
\end{enumerate}
For instance, $B_{i+1, j}$ (for $j\neq i,i+1$) has $\tau(B_{i+1,j}) = B_{ij}$, which satisfies the relations of Case $1$: 
\[[A_{i+1,i},B_{ij}] = B_{i+1,j},\quad [A_{i+1,i},B_{i+1,j}] = 0,\quad [h_i, B_{ij}] = B_{ij},\quad [A_{i,i+1},B_{ij}] = 0.\]
The information encoded in these case is enough to complete the computation. For instance, in Case $1$, we have \begin{align*}
\breve{q}_i(\bbar{x}) &= \tau_\gamma(x) + \frac{1}{h_\gamma} [e_\gamma,[e_{-\gamma}, \tau_\gamma(x)]] + \mathbb{I} \\
&= \tau_\gamma(x) + \frac{1}{h_\gamma} ([[e_\gamma,e_{-\gamma}],\tau_\gamma(x)] + [e_{-\gamma},[e_{\gamma},\tau_\gamma(x)]]) + \mathbb{I}\\
&=  \tau_\gamma(x) + \frac{1}{h_\gamma} [h_\gamma,\tau_\gamma(x)] + \mathbb{I}\\
&= \tau_\gamma(x) + \frac{1}{h_\gamma}\tau_\gamma(x) + \mathbb{I}\\
&= \frac{h_\gamma+1}{h_\gamma}\tau_\gamma(x) + \mathbb{I}.
\end{align*}
The remaining cases similarly reduce to a computation in $\mathfrak{p_{2n}}$.
\end{proof}

\subsection{The Structure of \texorpdfstring{$\mathcal{B}_n$}{Bn}} \label{sec: structure of Bn}
Following Theorem \ref{thm: generators and ordering relations}, $\mathcal{B}_n$ is generated as a $U'(\mathfrak{h}$)-module by $\mathcal{X} :=\set{\bbar{B}_{ij}}_{1\leq i\leq j\leq n}$. Moreover, they are related by \emph{ordering relations}; that is, relations of the form \begin{equation}
\dprod{y}{x} = \dprod{x}{y} + \sum \dprod{x_i}{y_i} + \sum\bbar{z_i},
\end{equation}
where $x < y$ and $x_i < y_i$ in some ordering compatible with the partial order on weights, and $x,y,x_i,y_i,z_i$ are elements of $\mathcal{X}$. For $i \leq j$ and $k\leq \ell$, we impose the order \(\bbar{B}_{ij} < \bbar{B}_{k\ell}\) when $ij < k\ell$. Note, the reverse of this order is compatible with the partial order on the weights. We show there is an order reversing anti-isomorphism of $\mathcal{B}_n$ in \ref{rem: Order reverse}.

There are only $17$ diamond product pairs that we compute directly to find relations. The remaining defining relations are obtained via Zhelobenko automorphisms. \\
For computations, we use the recurrence relation for $P_{\gl_n}$ discussed in Section $5.1$ of \cite{Kho2011}:
\begin{equation}P = \sum_{\lambda\in Q^+} J_\lambda,\end{equation} 
where each $J_\lambda$ has the form
\begin{equation}J_\lambda = \sum_i H_{i,\lambda} F_{i,\lambda} E_{i,\lambda}\end{equation}
such that each $H_{i,\lambda}\in U'(\mathfrak{h})$, $F_{i,\lambda}\in U(\mathfrak{n}_-)$ with weight $-\lambda$, and $E_{i,\lambda}\in U(\mathfrak{n}_+)$ with weight $\lambda$. One then defines $T_\gamma:T(U'(\mathfrak{gl}_n))\to T(U'(\mathfrak{gl}_n))$ for $\gamma\in \Delta_+$ by
\begin{equation}\label{eq:T-def}
T_\gamma\Big(\sum_i H_iF_iE_i\Big) = \sum_i H_iF_ie_{-\gamma} e_{\gamma}E_i
\end{equation}
for $H_i\in U'(\mathfrak{h})$, $F_i\in U(\mathfrak{n}_-)$ and $E_i\in U(\mathfrak{n}_+)$. Then $P$ is given recursively by $J_0 = 1$ and \begin{equation}J_\lambda = -\frac{1}{\tilde{h}_\lambda}\sum_{\gamma\in \Delta_+} T_\gamma(J_{\lambda-\gamma})\end{equation} for all $\lambda\in Q^+$, where \begin{equation}\tilde{h}_\lambda = h_\gamma + \frac{(\lambda,\lambda)}{2} + \operatorname{ht}(\lambda),\end{equation} and $h_\lambda = \sum_i m_i A_{ii}^S$ when $\lambda = \sum_i m_i \varepsilon_i$, and by convention $J_\mu=0$ if $\mu\notin Q^+$. We use this to compute several ordering relations for $\mathcal{B}_n$. \\

\begin{proposition} \label{prop: initial rels 1}
The following relations hold in $\mathcal{B}_n$.\begin{align}
\label{rel: 1a} \dprod{B_{12}}{B_{11}} &= \parens{1-\frac{2}{h_{\alpha_1}+2}}\dprod{B_{11}}{B_{12}} \tag{1a}\\
\label{rel: 2} \dprod{B_{22}}{B_{11}} &= \parens{1+\frac{2}{h_{\alpha_1}(h_{\alpha_1}+3)}}\parens{\dprod{B_{11}}{B_{22}}-\frac{1}{h_{\alpha_1}}\dprod{B_{12}}{B_{12}}} \tag{2} \\
\label{rel: 3a} \dprod{B_{23}}{B_{11}} &= \parens{1+\frac{2}{h_{\alpha_1 + \alpha_2}(h_{\alpha_1+\alpha_2}+3)}} \nonumber\\
&\phantom{=}\cdot \parens{\dprod{B_{11}}{B_{23}} - \frac{(h_{\alpha_2}+1)(h_{2\alpha_1+\alpha_2}+2)}{h_{\alpha_1}(h_{\alpha_2}+2)(h_{\alpha_1+\alpha_2}+1)} \dprod{B_{12}}{B_{13}}} \tag{3a} \\
\label{rel: 3b} \dprod{B_{22}}{B_{13}} &= \parens{1 + \frac{2}{h_{\alpha_1}(h_{\alpha_1}+3)}}\bigg(\parens{1 - \frac{2}{h_{\alpha_2}(h_{\alpha_2}+1)}}\dprod{B_{13}}{B_{22}} \nonumber\\
&\phantom{=}+ \parens{1-\frac{1}{h_{\alpha_1 + \alpha_2}+3}}\parens{\frac{1}{h_{\alpha_2}} - \frac{1}{h_{\alpha_1}+1}}\dprod{B_{12}}{B_{23}}\bigg) \tag{3b} \\
\label{rel: 4a} \dprod{B_{13}}{B_{12}} &= \parens{1-\frac{2}{h_{\alpha_1}(h_{\alpha_1}+1)}}\parens{1-\frac{1}{h_{\alpha_2}+2}}\parens{1+\frac{2}{h_{\alpha_1+\alpha_2}(h_{\alpha_1+\alpha_2}+3)}}\dprod{B_{12}}{B_{13}} \nonumber\\
&\phantom{=} + \frac{2(h_{\alpha_2}+1)(h_{\alpha_1+\alpha_2}+1)}{(h_{\alpha_1}+1)h_{\alpha_1+\alpha_2}(h_{\alpha_1+\alpha_2}+3)}\dprod{B_{11}}{B_{23}} \tag{4a} \\
\label{rel: 4b} \dprod{B_{23}}{B_{12}} &= \parens{1+\frac{1}{h_{\alpha_1}(h_{\alpha_1}+3)}}\parens{{\bigg (}1-\frac{1}{h_{\alpha_1 + \alpha_2}+3}}\dprod{B_{12}}{B_{23}} \nonumber \\
&\phantom{=} -\frac{2(h_{\alpha_1+\alpha_2}+2)}{(h_{\alpha_1}+2)(h_{\alpha_2}+1)}\dprod{B_{13}}{B_{22}}{\bigg )} \tag{4b} \\
\label{rel: 5a} \dprod{B_{34}}{B_{12}} &= \parens{1+\frac{1}{(h_{\alpha_1+\alpha_2}+1)(h_{\alpha_1+\alpha_2}+3)}}\parens{1+\frac{1}{(h_{\alpha_1+\alpha_2+\alpha_3}+2)(h_{\alpha_1+\alpha_2+\alpha_3}+4)}} \nonumber\\
&\phantom{=} \cdot{\bigg(} \dprod{B_{12}}{B_{34}}
 - \frac{(h_{\alpha_1}+1)(h_{\alpha_1+\alpha_2}+3)(h_{\alpha_1 + 2\alpha_2 + \alpha_3}+4)}{(h_{\alpha_1}+2)(h_{\alpha_1 + \alpha_2}+2)(h_{\alpha_2}+1)(h_{\alpha_1 + \alpha_2 + \alpha_3} + 3)} \dprod{B_{13}}{B_{24}} \nonumber \\
&\phantom{=} - \frac{(h_{\alpha_1}+1)h_{\alpha_3}(h_{\alpha_1 + \alpha_2 + \alpha_3}+4)(h_{\alpha_1 + 2\alpha_2 + \alpha_3}+4)}{(h_{\alpha_1} + 2)(h_{\alpha_3}+1)(h_{\alpha_1 + \alpha_2 + \alpha_3}+3)(h_{\alpha_1 + \alpha_2}+2)(h_{\alpha_2 + \alpha_3}+2)} \dprod{B_{14}}{B_{23}}{\bigg )} \tag{5a}
\end{align}
\end{proposition} 
\begin{proof}
We directly compute the diamond product of all relevant pairs, listed here. \begin{align*}
\dprod{B_{12}}{B_{11}} &= B_{12}B_{11} + \mathbb{I} \\
\dprod{B_{11}}{B_{12}} &= B_{11}B_{12} + 2(\tilde{h}_{\alpha_1}-2)^{-1}B_{12}B_{11} + \mathbb{I} \\
\dprod{B_{22}}{B_{11}} &= B_{22}B_{11} + \mathbb{I} \\
\dprod{B_{11}}{B_{22}} &= B_{11}B_{22} +(\tilde{h}_{\alpha_1}-2)^{-1}B_{12}B_{12} + 4(\tilde{h}_{2\alpha_1}-4)^{-1}(\tilde{h}_{\alpha_1}-2)^{-1}B_{22}B_{11} + \mathbb{I} \\
\dprod{B_{12}}{B_{12}} &= B_{12}B_{12} + 4\tilde{h}_{\alpha_1}^{-1}B_{22}B_{11} + \mathbb{I} \\
\dprod{B_{23}}{B_{11}} &= B_{23}B_{11} + \mathbb{I} \\
\dprod{B_{13}}{B_{12}} &= B_{13}B_{12} + 2(\tilde{h}_{\alpha_1}-1)^{-1} B_{23}B_{11} + \mathbb{I} \\
\dprod{B_{11}}{B_{23}} &= B_{11}B_{23} + (\tilde{h}_{\alpha_1}-2)^{-1}B_{12}B_{13} + (\tilde{h}_{\alpha_1 + \alpha_2}-2)^{-1}B_{13}B_{12} \\ &\phantom{=} + (\tilde{h}_{\alpha_1 + \alpha_2}-2)^{-1}(\tilde{h}_{\alpha_1}-2)^{-1}B_{13}B_{12} + 2(\tilde{h}_{2\alpha_1 + \alpha_2} - 4)^{-1}(\tilde{h}_{\alpha_1 + \alpha_2}-2)^{-1}B_{23}B_{11} \\
&\phantom{=} + 2(\tilde{h}_{2\alpha_1 + \alpha_2}-4)^{-1}(\tilde{h}_{\alpha_1}-2)^{-1}B_{23}B_{11} \\ &\phantom{=} + 2(\tilde{h}_{2\alpha_1 + \alpha_2}-4)^{-1}(\tilde{h}_{\alpha_1 + \alpha_2}-2)^{-1}(\tilde{h}_{\alpha_1}-2)^{-1}B_{23}B_{11} + \mathbb{I} \\
\dprod{B_{12}}{B_{13}} &= B_{12}B_{13} + (\tilde{h}_{\alpha_2} - 1)^{-1} B_{13}B_{12} + 2(\tilde{h}_{\alpha_1 + \alpha_2}-1)^{-1} B_{23}B_{11} \\ &\phantom{=} + 2(\tilde{h}_{\alpha_1 + \alpha_2}-1)^{-1}(\tilde{h}_{\alpha_2}-1)^{-1} B_{23}B_{11} + \mathbb{I} \\
\dprod{B_{23}}{B_{12}} &= B_{23}B_{12} + \mathbb{I} \\
\dprod{B_{22}}{B_{13}} &= B_{22}B_{13} + (\tilde{h}_{\alpha_2}-2)^{-1}B_{23}B_{12} + \mathbb{I} \\
\dprod{B_{13}}{B_{22}} &= B_{13}B_{22} + (\tilde{h}_{\alpha_1}-1)^{-1}B_{23}B_{12} + \mathbb{I} \\
\dprod{B_{12}}{B_{23}} &= B_{12}B_{23} + 2\tilde{h}_{\alpha_1}^{-1}B_{22}B_{13} + 2(\tilde{h}_{\alpha_2}-1)^{-1}B_{13}B_{22} + (\tilde{h}_{\alpha_1 + \alpha_2}-1)^{-1}(1 + 2\tilde{h}_{\alpha_1}^{-1} + 2(\tilde{h}_{\alpha_2}-1)^{-1}) B_{23}B_{12}\mathbb{I} \\
\dprod{B_{34}}{B_{12}} &= B_{34}B_{12} + \mathbb{I} \\
\dprod{B_{12}}{B_{34}} &= B_{12}B_{34} + (\tilde{h}_{\alpha_2}-1)^{-1} B_{13}B_{24} + (\tilde{h}_{\alpha_1 + \alpha_2}-1)^{-1}(1+(\tilde{h}_{\alpha_2}-1)^{-1}) B_{23}B_{14} \\ &\phantom{=} + (\tilde{h}_{\alpha_2 + \alpha_3}-1)^{-1}(1+(\tilde{h}_{\alpha_2}-1)^{-1}) B_{14}B_{23} \\ &\phantom{=} + (\tilde{h}_{\alpha_1 + \alpha_2 + \alpha_3}-1)^{-1}(1+((\tilde{h}_{\alpha_1 + \alpha_2}-1)^{-1}+(\tilde{h}_{\alpha_2 + \alpha_3}-1)^{-1})(1+(\tilde{h}_{\alpha_2}-1)^{-1})) B_{24}B_{13} \\ &\phantom{=} + (\tilde{h}_{\alpha_1 + 2\alpha_2 + \alpha_3}-2)^{-1}((\tilde{h}_{\alpha_1 + \alpha_2 + \alpha_3}-1)^{-1} + (\tilde{h}_{\alpha_2}-1)^{-1} \\ &\phantom{=} + (1+(\tilde{h}_{\alpha_1 + \alpha_2 + \alpha_3}-1)^{-1})((\tilde{h}_{\alpha_2 + \alpha_3}-1)^{-1} + (\tilde{h}_{\alpha_1 + \alpha_2}-1)^{-1})(1+(\tilde{h}_{\alpha_2}-1)^{-1})) B_{34}B_{12} + \mathbb{I} \\
\dprod{B_{13}}{B_{24}} &= B_{13}B_{24} + (\tilde{h}_{\alpha_1}-1)^{-1} B_{23}B_{14} + (\tilde{h}_{\alpha_3}-1)^{-1} B_{14}B_{23} \\
&\phantom{=} + (\tilde{h}_{\alpha_1 + \alpha_3}-2)^{-1}((\tilde{h}_{\alpha_1}-1)^{-1} + (\tilde{h}_{\alpha_3}-1)^{-1})B_{24}B_{13} \\
&\phantom{=} + (\tilde{h}_{\alpha_1 + \alpha_2 + \alpha_3}-1)^{-1}(1 + (1+(\tilde{h}_{\alpha_1 + \alpha_3}-2)^{-1})((\tilde{h}_{\alpha_1}-1)^{-1} + (\tilde{h}_{\alpha_3}-1)^{-1})B_{34}B_{12} + \mathbb{I} \\ 
\dprod{B_{14}}{B_{23}} &= B_{14}B_{23} + (\tilde{h}_{\alpha_1}-1)^{-1}B_{24}B_{13} + (\tilde{h}_{\alpha_1 + \alpha_2}-1)^{-1}(1+(\tilde{h}_{\alpha_1}-1)^{-1})B_{34}B_{12} + \mathbb{I}
\end{align*} This gives several systems of at most three equations that can be solved for misordered diamond products. 
As an example, we work through the computation for the relation of $\dprod{B_{13}}{B_{12}}$. The weight of this pair is $(\varepsilon_1 + \varepsilon_3) + (\varepsilon_1 + \varepsilon_2) = 2\varepsilon_1 + \varepsilon_2 + \varepsilon_3$. We compute the diamond products for other pairs of the same weight, in the desired order; that is, $\dprod{B_{12}}{B_{13}}$ and $\dprod{B_{11}}{B_{23}}$. For each pair $\dprod{x}{y}$, this amounts to finding all possible sequences of positive roots $\gamma_1,\cdots, \gamma_n$ for which $\opad e_{\gamma_n}\cdots \opad e_{\gamma_1} (y)$ and $\opad e_{-\gamma_n}\cdots \opad e_{-\gamma_1} (x)$ are both nonzero. We call such a sequence \textit{admissible}. \\
The product $\dprod{B_{13}}{B_{12}}$ is straightforward, since the only admissible sequences are $\varnothing$ and $\alpha_1$. This gives \[\dprod{B_{13}}{B_{12}} = B_{13}(1 - \tilde{h}_{\alpha_1}^{-1}e_{-\alpha_1}e_{\alpha_1})B_{12} + \mathbb{I}\] Ideally every term should be simplified, so that an element of $U'(\mathfrak{h})$ occurs on the left, followed by a quadratic term in $B's$. In this case, \[\dprod{B_{13}}{B_{12}} = B_{12}B_{13} + 2(\tilde{h}_{\alpha_1}-1)^{-1}B_{11}B_{23}.\]
Slightly more involved, we repeat this process for $\dprod{B_{12}}{B_{13}}$. The admissible sequences of positive roots are $\varnothing$; $\alpha_2$; $\alpha_1 + \alpha_2$; and $\alpha_2,\alpha_1$. Then we have \[\dprod{B_{12}}{B_{13}} = B_{12}(1 - \tilde{h}_{\alpha_2}^{-1} e_{-\alpha_2}e_{\alpha_2} - \tilde{h}_{\alpha_1 + \alpha_2}^{-1}e_{-(\alpha_1 + \alpha_2)}e_{\alpha_1 + \alpha_2} + \tilde{h}_{\alpha_1+\alpha_2}^{-1}\tilde{h}_{\alpha_2}^{-1}e_{-\alpha_2}e_{-\alpha_1}e_{\alpha_1}e_{\alpha_2})B_{13} + \mathbb{I}\] Reordering each term and consolidating, \[\dprod{B_{12}}{B_{13}} = (1+(\tilde{h}_{\alpha_2}-1)^{-1})(B_{12}B_{13} + 2(\tilde{h}_{\alpha_1 + \alpha_2}-1)^{-1}B_{11}B_{23}) + \mathbb{I}.\]
Lastly, consider $\dprod{B_{11}}{B_{23}}$. The admissible sequences of positive roots are \[\varnothing;\; \alpha_1;\; \alpha_1+\alpha_2;\; \alpha_1,\alpha_2;\; \alpha_1 + \alpha_2,\alpha_1;\; \alpha_1,\alpha_1+\alpha_2;\; \alpha_1,\alpha_2,\alpha_1.\] The product is then given by \begin{align*}
\dprod{B_{11}}{B_{23}} &= B_{11}(1 - \tilde{h}_{\alpha_1}^{-1}e_{-\alpha_1}e_{\alpha_1} - \tilde{h}_{\alpha_1 + \alpha_2}^{-1}(e_{-(\alpha_1+\alpha_2)}e_{\alpha_1 + \alpha_2} - \tilde{h}_{\alpha_1}^{-1}e_{-\alpha_1}e_{-\alpha_2}e_{\alpha_2}e_{\alpha_1}) \\
&\phantom{=} -\tilde{h}_{2\alpha_1 + \alpha_2}^{-1}(-\tilde{h}_{\alpha_1+\alpha_2}^{-1}(e_{-(\alpha_1 + \alpha_2)}e_{-\alpha_1}e_{\alpha_1}e_{\alpha_1+\alpha_2} - \tilde{h}_{\alpha_1}e_{-\alpha_1}e_{-\alpha_2}e_{-\alpha_1}e_{\alpha_1}e_{\alpha_2}e_{\alpha_1}) \\ &\phantom{=} \phantom{-\tilde{h}_{2\alpha_1+\alpha_2}^{-1}(}-\tilde{h}_{\alpha_1}^{-1}e_{-\alpha_1}e_{-(\alpha_1 + \alpha_2)}e_{\alpha_1 + \alpha_2}e_{\alpha_1}))B_{23} + \mathbb{I}
\end{align*}
Equivalently, \begin{align*}
\dprod{B_{11}}{B_{23}} &= \parens{1 + 2(\tilde{h}_{2\alpha_1 + \alpha_2}-4)^{-1}\parens{(\tilde{h}_{\alpha_1 +\alpha_2}-2)^{-1}(1+(\tilde{h}_{\alpha_1}-2)^{-1}) + (\tilde{h}_{\alpha_1}-2)^{-1}}}B_{11}B_{23} \\
&\phantom{=} +\parens{(\tilde{h}_{\alpha_1}-2)^{-1} + (\tilde{h}_{\alpha_1 +\alpha_2}-2)^{-1}(1+(\tilde{h}_{\alpha_1}-2)^{-1})}B_{12}B_{13} + \mathbb{I}
\end{align*}
We restate these products in a simplified form: \begin{align*}
\dprod{B_{13}}{B_{12}} &= \frac{2}{h_{\alpha_1}+1} B_{11}B_{23} + B_{12}B_{13} + \mathbb{I} \\
\dprod{B_{12}}{B_{13}} &= \parens{1+\frac{1}{h_{\alpha_2}+1}}\parens{\frac{2}{h_{\alpha_1+\alpha_2}+2}B_{11}B_{23} + B_{12}B_{13}} + \mathbb{I} \\
\dprod{B_{11}}{B_{23}} &= \parens{1 + \frac{2}{h_{\alpha_1}(h_{\alpha_1 + \alpha_2}+1)}}B_{11}B_{23} + \frac{h_{2\alpha_1 + \alpha_2}+2}{h_{\alpha_1}(h_{\alpha_1 + \alpha_2}+1)}B_{12}B_{13} +\mathbb{I} \\
\end{align*} For the sake of illustration, we will write this as \begin{align*}
\dprod{B_{13}}{B_{12}} &= x_1 B_{11}B_{23} + x_2 B_{12}B_{13}\\
\dprod{B_{12}}{B_{13}} &= y_1 B_{11}B_{23} + y_2 B_{12}B_{13} \\
\dprod{B_{11}}{B_{23}} &= z_1 B_{11}B_{23} + z_2 B_{12}B_{13} 
\end{align*} As long as $y_1z_2-y_2z_1$ is invertible in $U'(\mathfrak{h})$, we can apply Cramer's rule to obtain \[\dprod{B_{13}}{B_{12}} = \frac{1}{y_1z_2-y_2z_1}\parens{(x_1z_2-x_2z_1)\dprod{B_{12}}{B_{13}} - (x_1y_2-x_2y_1)\dprod{B_{11}}{B_{23}}}\] This is indeed the case, since \[y_1z_2-y_2z_1 = -\frac{(h_{\alpha_2}+1)(h_{\alpha_1 + \alpha_2} + 1)(h_{\alpha_1 + \alpha_2} + 2)}{(h_{\alpha_2}+2)h_{\alpha_1 + \alpha_2}(h_{\alpha_1 + \alpha_2} + 3)}\] The other relations are obtained by solving similar systems of equations. \\
\end{proof}

From these seven ordering relations, we obtain five more using only Zhelobenko automorphisms.
\begin{proposition} \label{prop: initial rels 2}
The following relations hold in $\mathcal{B}_n$.
\begin{align}
\label{rel: 1b} \dprod{B_{22}}{B_{12}} &= \parens{1-\frac{2}{h_{\alpha_1}+4}}\dprod{B_{12}}{B_{22}} \tag{1b}\\
\label{rel: 3c} \dprod{B_{33}}{B_{12}} &= \parens{1 + \frac{2}{(h_{\alpha_1 + \alpha_2}+1)(h_{\alpha_1 + \alpha_2}+4)}}\bigg(\dprod{B_{12}}{B_{33}} \nonumber\\ &\phantom{=}- \parens{1-\frac{1}{h_{\alpha_1} + 2}}\parens{1 + \frac{1}{h_{\alpha_2} + 1}}\parens{\frac{1}{h_{\alpha_2}+2}+\frac{1}{h_{\alpha_1 + \alpha_2}+2}}\dprod{B_{13}}{B_{23}}\bigg) \tag{3c} \\
\label{rel: 4c} \dprod{B_{23}}{B_{13}} &= \parens{1+\frac{2}{(h_{\alpha_1+\alpha_2}+1)(h_{\alpha_1+\alpha_2}+4)}}\bigg(\parens{1-\frac{1}{h_{\alpha_1}+2}} \nonumber \\
&\phantom{=}\cdot\parens{1-\frac{2}{(h_{\alpha_2}+1)(h_{\alpha_2}+2)}}\dprod{B_{13}}{B_{23}} + \frac{2(h_{\alpha_1}+1)}{(h_{\alpha_2}+2)(h_{\alpha_1 + \alpha_2}+3)}\dprod{B_{12}}{B_{33}}\bigg) \tag{4c} \\
\label{rel: 5b} \dprod{B_{24}}{B_{13}} &= \parens{1 + \frac{1}{h_{\alpha_1}(h_{\alpha_1}+2)}}\parens{1 + \frac{1}{(h_{\alpha_1 + \alpha_2 + \alpha_3}+2)(h_{\alpha_1 + \alpha_2 + \alpha_3}+4)}} \nonumber \\
&\phantom{=}\cdot\bigg(\parens{1-\frac{1}{(h_{\alpha_2}+1)^2}}\dprod{B_{13}}{B_{24}} \nonumber \\
&\phantom{=} + \frac{(h_{\alpha_1}+2)(h_{\alpha_1+\alpha_2} + 2)(h_{\alpha_1 + \alpha_3}+2)}{(h_{\alpha_1}+1)(h_{\alpha_1+\alpha_2} + 3)(h_{\alpha_2}+1)(h_{\alpha_1 + \alpha_2 + \alpha_3}+3)}\dprod{B_{12}}{B_{34}} \nonumber \\ 
&\phantom{=} - \frac{(h_{\alpha_2}+2)(h_{\alpha_1 + \alpha_2}+2)(h_{\alpha_2 + \alpha_3}+1)(h_{\alpha_1+\alpha_2+\alpha_3}+4)(h_{\alpha_1 + \alpha_3}+2)}{(h_{\alpha_2}+1)(h_{\alpha_1 + \alpha_2}+3)(h_{\alpha_2 + \alpha_3}+2)(h_{\alpha_1+\alpha_2+\alpha_3}+3)(h_{\alpha_1}+1)(h_{\alpha_3}+1)}\dprod{B_{14}}{B_{23}}\bigg) \tag{5b} \\
\label{rel: 5c} \dprod{B_{23}}{B_{14}} &= \parens{1 + \frac{1}{h_{\alpha_1}(h_{\alpha_1}+2)}}\parens{1 + \frac{1}{(h_{\alpha_1+\alpha_2}+1)(h_{\alpha_1+\alpha_2}+3)}}
\nonumber \\
&\phantom{=} \cdot\bigg(\parens{1-\frac{1}{(h_{\alpha_2+\alpha_3}+2)^2}}\parens{1-\frac{1}{(h_{\alpha_3}+1)^2}}\dprod{B_{14}}{B_{23}} \nonumber\\ &\phantom{=} + \frac{(h_{\alpha_1}+2)(h_{\alpha_3}+2)(h_{\alpha_1+\alpha_2+\alpha_3}+3)(h_{\alpha_1}-h_{\alpha_3})}{(h_{\alpha_1}+1)(h_{\alpha_3}+1)(h_{\alpha_1+\alpha_2+\alpha_3}+4)(h_{\alpha_2 + \alpha_3}+2)(h_{\alpha_1 + \alpha_2}+2)}\dprod{B_{12}}{B_{34}} \nonumber\\
&\phantom{=} + \frac{h_{\alpha_2}(h_{\alpha_1 + \alpha_2}+3)(h_{\alpha_2+\alpha_3}+3)(h_{\alpha_1 + \alpha_2 + \alpha_3}+3)(h_{\alpha_1}-h_{\alpha_3})}{(h_{\alpha_2}+1)(h_{\alpha_1 + \alpha_2}+2)(h_{\alpha_2+\alpha_3}+2)(h_{\alpha_1 + \alpha_2 + \alpha_3}+4)(h_{\alpha_1}+1)(h_{\alpha_3}+1)}\dprod{B_{13}}{B_{24}} \bigg) \tag{5c}
\end{align}
\end{proposition}
\begin{proof}
The relation \eqref{rel: 1b} is obtained by applying $\breve{q}_1$ to the relation \eqref{rel: 1a} from Proposition \ref{prop: initial rels 1}. We show this case in detail. Lemma \ref{lem: zhel auto Bn} provides a summary of the behavior of $\breve{q}_1$. In particular, \begin{equation}
\breve{q}_1(\bbar{B}_{12}) = -\frac{h_1+2}{h_1}B_{12} + \mathbb{I}, \quad \breve{q}_1(\bbar{B}_{11}) = \bbar{B}_{22},\quad \breve{q}_1(h_1) = -(h_1 + 2).
\end{equation}Applying this to relation \eqref{rel: 1a} we obtain \begin{equation}
\parens{-\frac{h_1+2}{h_1}B_{12} + \mathbb{I}}\diamond (\bbar{B}_{22}) = \frac{h_1 + 2}{h_1}(\bbar{B}_{22})\diamond \parens{-\frac{h_1+2}{h_1}B_{12} + \mathbb{I}}
\end{equation} Consolidating elements of $U'(\mathfrak{h})$ to the left, we have \begin{equation}
-\frac{h_1 + 2}{h_1}\dprod{B_{12}}{B_{22}} = -\frac{(h_1 + 2)(h_1 + 4)}{h_1(h_1 + 2)}\dprod{B_{22}}{B_{12}} 
\end{equation}
Lastly we isolate the misordered term to obtain the ordering relation \begin{equation}
\dprod{B_{22}}{B_{12}} = \parens{1-\frac{2}{h_{\alpha_1}+4}}\dprod{B_{12}}{B_{22}}
\end{equation}
Similarly, applying $\breve{q}_2$ to relations \eqref{rel: 3b}, \eqref{rel: 4b} and \eqref{rel: 5a} give relations \eqref{rel: 3c}, \eqref{rel: 4c}, and \eqref{rel: 5b}, respectively. Applying $\breve{q}_3$ to \eqref{rel: 5b} gives \eqref{rel: 5c}.
\end{proof}

We refer to the relations listed in Propositions \ref{prop: initial rels 1} and \ref{prop: initial rels 2} as \textit{initial relations}. They are organized by type based on the misordered term on the left-hand side. The following table clarifies each type. For $i < j < k < \ell$:
\renewcommand{\arraystretch}{1.2}
\begin{table}[ht!]
\centering
\begin{tabular}{|c|c|}
\hline
Misordered Term & Type \\
\hline
$\dprod{B_{ij}}{B_{ii}}$ & \eqref{rel: 1a} \\
\hline
$\dprod{B_{jj}}{B_{ij}}$ & \eqref{rel: 1b} \\
\hline
$\dprod{B_{jj}}{B_{ii}}$ & \eqref{rel: 2} \\
\hline
$\dprod{B_{jk}}{B_{ii}}$ & \eqref{rel: 3a} \\
\hline
$\dprod{B_{jj}}{B_{ik}}$ & \eqref{rel: 3b} \\
\hline
$\dprod{B_{kk}}{B_{ij}}$ & \eqref{rel: 3c} \\
\hline
$\dprod{B_{ik}}{B_{ij}}$ & \eqref{rel: 4a} \\
\hline
$\dprod{B_{jk}}{B_{ij}}$ & \eqref{rel: 4b} \\
\hline
$\dprod{B_{jk}}{B_{ik}}$ & \eqref{rel: 4c} \\
\hline
$\dprod{B_{kl}}{B_{ij}}$ & \eqref{rel: 5a} \\
\hline
$\dprod{B_{jl}}{B_{ik}}$ & \eqref{rel: 5b} \\
\hline
$\dprod{B_{jk}}{B_{il}}$ & \eqref{rel: 5c} \\
\hline
\end{tabular}
\caption{Types of Ordering Relations in $\mathcal{B}_n$ $(i < j < k < \ell)$}
\label{tab: rel types}
\end{table}
\renewcommand{\arraystretch}{1}

\begin{theorem} \label{thm: All relations for Bn}
All ordering relations of $\mathcal{B}_n$ can be obtained from initial relations via Zhelobenko automorphisms. Specifically, for $i < j < k < l$ the relation that corresponds to permuting the indices of $B$'s by $1\mapsto i$, $2\mapsto j$, $3\mapsto k$, $4\mapsto\ell$ can be done directly, with the only other difference being the substitution \begin{align*}
h_{\alpha_1} &\mapsto A_{ii} - A_{jj} + (j-i - 1) \\
h_{\alpha_2} &\mapsto A_{jj} - A_{kk} + (k-j - 1) \\
h_{\alpha_3} &\mapsto A_{kk} - A_{\ell\ell} + (\ell-k - 1)
\end{align*}
\end{theorem}
\begin{proof} An important result of \ref{lem: zhel auto Bn} is that $\breve{q}_i(\bbar{B}_{mn}) = \tau(B_{mn}) + \mathbb{I}$ whenever $m,n\neq i+1$. Consequently, as long as $m,n \leq 4$ and $i<j<k<\ell$, we have \begin{equation}
\breve{q}_{i-1}\cdots \breve{q}_1 \breve{q}_{j-1}\cdots \breve{q}_{2}\breve{q}_{k-1}\cdots\breve{q}_{3}\breve{q}_{\ell-1}\cdots\breve{q}_4(\bbar{B}_{mn}) = \tau_{i-1}\cdots\tau_1 \tau_{j-1}\cdots \tau_{2}\tau_{k-1}\cdots\tau_{3}\tau_{\ell-1}\cdots\tau_4(B_{mn}) + \mathbb{I}
\end{equation}
Applying this map to an initial relation therefore has the effect of permuting the indices of $B$'s as claimed. The only thing to check is the effect of this map on elements of $U'(\mathfrak{h})$. By Lemma \ref{lem: zhel auto Bn} we find that for $i < j$, \[\breve{q}_{j-1}\cdots \breve{q}_i(A_{ii}) = A_{jj} - (j-i).\] We also have that for $j\neq i,i+1$, \[\breve{q}_j(A_{ii}) = A_{ii}.\] Combining these facts it is straightforward to see the claim holds.
\end{proof}

\begin{example} \label{ex: all relations example}
For $n\geq 6$, the ordering relation for $\dprod{B_{66}}{B_{33}}$ is of type \eqref{rel: 2}, and has initial relation
\[
\dprod{B_{22}}{B_{11}} = \parens{1+\frac{2}{h_{\alpha_1}(h_{\alpha_1}+3)}}\parens{\dprod{B_{11}}{B_{22}}-\frac{1}{h_{\alpha_1}}\dprod{B_{12}}{B_{12}}} \]

Then the ordering relation for $\dprod{B_{66}}{B_{33}}$ is obtained by applying $\breve{q}_2\breve{q}_1\breve{q}_5\breve{q}_4\breve{q}_3\breve{q}_2$ to the initial relation, giving
\[
\dprod{B_{66}}{B_{33}} = \parens{1+\frac{2}{(h_{\alpha_3+\alpha_4+\alpha_5}+2)(h_{\alpha_3+\alpha_4+\alpha_5}+5)}}\parens{\dprod{B_{33}}{B_{66}}-\frac{1}{h_{\alpha_3+\alpha_4+\alpha_5}+2}\dprod{B_{36}}{B_{36}}} \]
\end{example}

\section{The Diagonal Reduction Algebra \texorpdfstring{$\mathcal{Z}(\sp_{2n}\times\sp_{2n}, \sp_{2n})$}{Z(sp2nxsp2n,sp2n)}} \label{sec: diag red alg sp2n}

We now consider the case when $\mathfrak{g}=\sp_{2n}\times \sp_{2n}$ and $\mathfrak{k}=\sp_{2n}$. We use the same basis for $\sp_{2n}$ as in Section \ref{sec: symplectic Lie alg}. For $x\in \sp_{2n}$ denote \begin{equation}x^S = (x,x) \quad\text{and} \quad x^L = (x,0).\end{equation}  The Lie subalgebra generated by diagonal elements is isomorphic to $\sp_{2n}$, while $\mathfrak{g}$ is generated by $\set{x^S,x^L}_{x\in \mathcal{X}}$, where $\mathcal{X} = \set{A_{ij},B_{ij},C_{ij}}$. The superscript `L' is often omitted. \\

We use the decomposition \[\mathfrak{k} = \mathfrak{n}_-\oplus \mathfrak{h}\oplus\mathfrak{n}_+\] where \begin{align}
\mathfrak{h} &= \spn_\mathbb{C}\{A_{ii}^S\}_{i=1}^n, \\
\mathfrak{n}_+ &= \spn_\mathbb{C}\parens{\set{A^S_{ij}}_{1\leq i<j\leq n}\cup \set{B^S_{ij}}_{1\leq i\leq j\leq n}} \\
\mathfrak{n}_- &= \spn_\mathbb{C}\parens{\set{A^S_{ji}}_{1\leq i<j\leq n}\cup \set{C^S_{ij}}_{1\leq i\leq j\leq n}}.
\end{align}
For $1\leq i\leq n$ define $\varepsilon_i\in \mathfrak{h}^*$ by $\varepsilon_i(A^S_{jj}) =\delta_{ij}$, extended linearly.  We choose the positive roots \begin{equation}\Delta_+ = \set{\varepsilon_i - \varepsilon_j}_{1\leq i < j\leq n}\cup \set{\varepsilon_i + \varepsilon_j}_{1\leq i \leq j \leq n}\end{equation} so that the set of simple roots is \begin{equation}\Pi = \set{\varepsilon_i - \varepsilon_{i+1}}_{1\leq i < n}\cup\set{2\varepsilon_n}.\end{equation} For $1\leq i,j\leq n$ define \begin{align}
e_{\varepsilon_i - \varepsilon_j} &= A^S_{ij} && i\neq j \\
e_{\varepsilon_i+\varepsilon_j} &= B^S_{ij} && i\leq j \\
e_{-\varepsilon_i - \varepsilon_j} &= C^S_{ij} && i\leq j \\
h_{\varepsilon_i \pm\varepsilon_j} &= A^S_{ii} \pm A^S_{jj} && i< j \\
h_{2\varepsilon_i} &= A^S_{ii}. && 
\end{align}

\subsection{Extremal Projector} \label{sec: extremal proj for Dsp2n}
We first require a normal ordering of $\Delta_+$.
\begin{proposition}
The lexicographic ordering of the positive root vectors corresponds to a normal ordering of $\Delta_+$.
\end{proposition}
\begin{proof}
The positive root associated with $A_{ij}$ with $i < j$ is $\varepsilon_i-\varepsilon_j$. The only way this could be the sum of two other positive roots is \begin{equation}\varepsilon_i - \varepsilon_j = (\varepsilon_i-\varepsilon_k) + (\varepsilon_k-\varepsilon_j)\end{equation} with $i < k < j$. The associated root vectors $A_{ik} < A_{ij} < A_{kj}$ in lexicographic order correctly places $\varepsilon_i - \varepsilon_j$ between its summands. \\
We use $(ij)$ to indicate the subscript $ij$ or $ji$, whichever is in nondecreasing order. The positive root associated with $B_{(ij)}$ is $\varepsilon_i + \varepsilon_j$. The only way this could be the sum of two other positive roots is \begin{equation}\varepsilon_i + \varepsilon_j = (\varepsilon_i - \varepsilon_k) + (\varepsilon_j  + \varepsilon_k)\end{equation} with $i < k$. The associated root vectors $A_{ik} < B_{(ij)} < B_{(jk)}$ in lexicographic order correctly places $\varepsilon_i + \varepsilon_j$ between its summands.
\end{proof}
 The Weyl vector is \begin{equation}\rho = \sum_{k=1}^n (n+1-k)\varepsilon_k.\end{equation} The extremal projector therefore has the form \begin{equation} \label{eq: extremal proj for Dsp2n} P = \prod_{1\leq i < j \leq n}^\text{lex} P_{\varepsilon_i-\varepsilon_j}[j-i] \prod_{1\leq i\leq j\leq n}^\text{lex} P_{\varepsilon_i + \varepsilon_j}\left[\frac{2(n+1)-(i+j)}{2^{\delta_{ij}}}\right],\end{equation} 
 where \begin{equation}P_\gamma[t] = \sum_{k\geq 0}\frac{(-1)^k}{k!}\varphi_{\gamma,k}[t] e^k_{-\gamma}e^k_\gamma\end{equation} and
 \begin{equation}\varphi_{\gamma,k}[t] = \prod_{\ell=1}^k (h_\gamma + t + \ell)^{-1}.\end{equation} 
 Note the similarity  between Equations \eqref{eq: extremal proj for Dsp2n} and \eqref{eq: extremal proj for Bn}.

We can then define the double coset algebra as in Section \ref{sec: double coset alg general}, which we call $\mathcal{D}(\sp_{2n}) := (U'(\sp_{2n}\times\sp_{2n}) / \mathbb{I}, \diamond)$.

\subsection{Zhelobenko Automorphisms} \label{sec: Zhelo autos Dsp2n case}

We name the simple roots $\alpha_i = \varepsilon_i - \varepsilon_{i+1}$ for $1\leq i < n$ and $\alpha_n = 2\varepsilon_n$. The subscript `$i$' is used in place of `$\alpha_i$'. \\
The action of the automorphisms $\tau_i$ on the the basis $\mathcal{X}=\set{A_{ij},B_{ij},C_{ij}}$ can be described by the following rules: \begin{itemize}
    \item For $1\leq i < n$, $\tau_i$ interchanges indices $i$ and $i+1$. The sign changes for every instance of $i$ replaced by $i+1$. \\
    \item For $i=n$, $\tau_n$ swaps $\varepsilon_n$ with $-\varepsilon_n$. The sign changes only for $A_{jn},$ $A_{nj}$ $(j\le n)$, $B_{nn}$ and $C_{nn}$.
\end{itemize}
\begin{example}
The case $\sp_6$ is sufficiently large to exhibit all of the possibilities above. We give explicitly the effect of each automorphism $\tau_i$.  \\
\begin{center}
\begin{tabular}{|c||c|c|c|}
\hline
$x$ & $\tau_1(x)$ & $\tau_2(x)$ & $\tau_3(x)$ \\
\hline
\hline
$A_{11}$ & $A_{22}$ & $A_{11}$ & $A_{11}$ \\
\hline
$A_{12}$ & $-A_{21}$ & $-A_{13}$ & $A_{12}$ \\
\hline
$A_{13}$ & $-A_{23}$ & $A_{12}$ & $-B_{13}$ \\
\hline
$A_{21}$ & $-A_{12}$ & $-A_{31}$ & $A_{21}$ \\
\hline
$A_{22}$ & $A_{11}$ & $A_{33}$ & $A_{22}$ \\
\hline
$A_{23}$ & $A_{13}$ & $-A_{32}$ & $-B_{23}$ \\
\hline
$A_{31}$ & $-A_{32}$ & $A_{21}$ & $-C_{13}$ \\
\hline
$A_{32}$ & $A_{31}$ & $-A_{23}$ & $-C_{23}$ \\
\hline
$A_{33}$ & $A_{11}$ & $A_{22}$ & $-A_{33}$ \\
\hline
$B_{11}$ & $B_{22}$ & $B_{11}$ & $B_{11}$ \\
\hline
$B_{12}$ & $-B_{12}$ & $-B_{13}$ & $B_{12}$ \\
\hline
$B_{13}$ & $-B_{23}$ & $B_{12}$ & $A_{13}$ \\
\hline
$B_{22}$ & $B_{11}$ & $B_{33}$ & $B_{22}$ \\
\hline
$B_{23}$ & $B_{13}$ & $-B_{23}$ & $A_{23}$ \\
\hline
$B_{33}$ & $B_{33}$ & $B_{22}$ & $-C_{33}$ \\
\hline
$C_{11}$ & $C_{22}$ & $C_{11}$ & $C_{11}$ \\
\hline
$C_{12}$ & $-C_{12}$ & $-C_{13}$ & $C_{12}$ \\
\hline
$C_{13}$ & $-C_{23}$ & $C_{12}$ & $A_{31}$ \\
\hline
$C_{22}$ & $C_{11}$ & $C_{33}$ & $C_{22}$ \\
\hline
$C_{23}$ & $C_{13}$ & $-C_{23}$ & $A_{32}$ \\
\hline
$C_{33}$ & $C_{33}$ & $C_{22}$ & $-B_{33}$ \\
\hline
\end{tabular}
\end{center}
\end{example}
\vspace{12pt}
We may now describe the Zhelobenko automorphisms. There are only four types of images that must be described. \begin{lemma} \label{lem: zhel auto for Dsp2n} We introduce notation $H_i = A_{ii} - A_{i+1,i+1}$ and $I_i = A_{ii} + A_{i+1,i+1}$ for $1\leq i < n$. For each $i<n$ we consider the image of $\mathcal{X}_i = (\mathcal{X}\cup\set{H_i,I_i})\setminus\set{A_{ii},A_{i+1,i+1}}$, and take $\mathcal{X}_n = \mathcal{X}$.
\begin{enumerate}[{\rm 1)}]
    \item $\breve{q}_i(\bbar{x}^L) = \frac{h_i+1}{h_i}\tau_i(x)^L + \mathbb{I}$ under any of the following conditions:
    \begin{itemize}
        \item $1\leq i < n$ and $x\in\set{ A_{i+1,j},A_{ji}, B_{i+1,j},C_{ij}}_{j\neq i,i+1}$; or
        \item $i = n$ and  $x\in\set{A_{jn}, C_{jn}}_{j\neq n}$
    \end{itemize}
    \item $\breve{q}_i(\bbar{x}^L) = \frac{h_i+2}{h_i}\tau_i(x)^L + \mathbb{I}$ under any of the following conditions: 
    \begin{itemize}
        \item $1\leq i < n$ and $x\in\set{H_i,B_{i,i+1},C_{i,i+1}}$; or
        \item $i = n$ and  $x = A_{nn}$.
    \end{itemize}
    \item $\breve{q}_i(\bbar{x}^L) = \frac{h_i + 1}{h_i-1}\tau_i(x)^L + \mathbb{I}$ under any of the following conditions: 
     \begin{itemize}
        \item $1\leq i < n$ and $x\in \set{A_{i+1,i},B_{i+1,i+1},C_{ii}}$; or
        \item $i = n$ and  $x= C_{nn}$.
    \end{itemize}
    \item $\breve{q}_i(\bbar{x}^L) = \tau_i(x)^L + \mathbb{I}$ for every other $x\in\mathcal{X}_i$.
\end{enumerate} All of this holds if the superscript 'L' is replaced with 'S'.
\end{lemma}

\begin{proof}

Refer to the proof of Lemma \ref{lem: zhel auto Bn}.
\end{proof}

\subsection{Alternate Realizations of  \texorpdfstring{$\mathcal{B}_n$}{Bn}} \label{sec: alt realziations}

We now explore other realizations of $\mathcal{B}_n$. On one hand, $\mathcal{B}_n$ is a special case of a reduction algebra $\mathcal{Z}(\gl_n\ltimes W,\gl_n)$ where $W$ is a $\gl_n$-module. On the other hand, $\mathcal{B}_n$ can be embedded into $\mathcal{D}(\sp_{2n})$.

\begin{theorem} \label{thm: alt realizations}

Let $\mathfrak{h}$ be the Cartan subalgebra of $\sp_{2n}$ and $U'(\mathfrak{h})$ be the localization of $U(\mathfrak{h})$, both as described in Section \ref{sec: diag red alg sp2n}. The following are isomorphic as $U'(\mathfrak{h})$-rings:
    
    \begin{enumerate}
        
        \item The $U'(\mathfrak{h})$-subring of $\mathcal{D}(\sp_{2n})$ generated by $\bbar{B}_{ij}$, ($1\le i\le j\le n$);
        \item The localization at $U'(\mathfrak{h})$ of $\mathcal{Z}(\mathfrak{p}_{2n},\mathfrak{gl}_n)$ where $\mathfrak{p}_{2n}$ is the parabolic subalgebra $\left\{\left(\begin{smallmatrix}A&B=B^\top\\0&-A^\top\end{smallmatrix}\right)\right\}$ of $\mathfrak{sp}_{2n}$;
        \item The localization at $U'(\mathfrak{h})$ of $\mathcal{Z}(\gl_n\ltimes W,\mathfrak{gl}_n)$ where $W=S^2(\mathbb{C}^n)$ is the symmetric square of the fundamental representation.
    \end{enumerate}
\end{theorem}

\begin{proof}
Observe $2$ and $3$ are equivalent because  $\mathfrak{p}_{2n}$ is isomorphic to $\gl_n\ltimes W$ as Lie algebras and $\gl_n$-modules via the map \[A_{ij}\mapsto (E_{ij},0),\quad B_{ij}\mapsto (0,e_ie_j/2^{\delta_{ij}}).\]

It remains to show $1$ and $3$ are equivalent. Let $\widetilde{\mathcal{B}}_n$ be the $U'(\mathfrak{h})$-subring of $\mathcal{D}(\sp_{2n})$ generated by $\set{\bbar{B}_{ij}^L}_{1\leq i\leq j\leq n}$. The key observation here is that regardless of shift $t$, \[P_{\varepsilon_i + \varepsilon_j}[t](\bbar{B}_{k\ell}^L) = B_{k\ell}^L + I_-\] because $B_{k\ell}$ commutes with $B_{ij}$ in $\sp_{2n}$ for any $i,j,k,\ell$. Thus in $\mathcal{D}(\sp_{2n})$ by \eqref{eq: extremal proj for Dsp2n} we have \[\dprod{B_{ij}^L}{B_{k\ell}^L} = B_{ij}^LP_{\sp_{2n}} B_{k\ell}^L+\mathbb{I} = B_{ij}^LP_{\gl_{n}}B_{k\ell}^L + \mathbb{I}.\] In other words, all the relations of Theorem \ref{thm: All relations for Bn} hold in $\mathcal{D}(\sp_{2n})$. By Theorem \ref{thm: generators and ordering relations} these relations fully define the structure of $\widetilde{\mathcal{B}}_n$. Hence the map $\mathcal{B}_n\to \widetilde{\mathcal{B}}_n$ given by $B_{ij}\mapsto B_{ij}^L$ and $A_{ii}\mapsto A_{ii}^S$ is injective homomorphism. The only difference between them is that $\widetilde{\mathcal{B}}_n$ has been localized further (for instance, the element $h_{\alpha_n}^{-1}$ is present in $\widetilde{\mathcal{B}}_n$, but has no analogue in $\mathcal{B}_n$). Localizing $\mathcal{B}_n$ at $U'(\mathfrak{h})$ fixes this issue and gives the desired isomorphism.
    
\end{proof}

\subsection{Chevalley Anti-Involution} \label{sec: Chev anti-inv}
The \textit{Chevalley anti-involution} $\epsilon: \sp_{2n}\to\sp_{2n}$ is an anti-automorphism of order $2$ that maps 
$A_{ij}\mapsto A_{ji}$, $B_{ij}\mapsto C_{ij}$, and $C_{ij}\to B_{ij}$. It can be extended to an anti-automorphism of $U'(\sp_{2n}\times \sp_{2n})$ so that $\epsilon(x^L) = \epsilon(x)^L$ and $\epsilon(x^S) = \epsilon(x)^S$. Like Zhelobenko automorphisms, the Chevalley anti-involution helps cut down the number of required computations.

\begin{remark} \label{rem: Order reverse}
The $7$ initial relations in $\mathcal{B}_n$ from which all other ordering relations were obtained has some redundancy when $\mathcal{B}_n$ is viewed as a subalgebra of $\mathcal{D}(\sp_{2n})$. Specifically, relations of type \eqref{rel: 3b} and \eqref{rel: 4b} are recoverable from \eqref{rel: 3a} and \eqref{rel: 4a}, respectively, if we also allow the use of Chevalley anti-involution $\epsilon$. This requires the map \[\varsigma:= \varepsilon(\breve{q}_n)(\breve{q}_{n-1}\breve{q}_n\breve{q}_{n-1})\cdots(\breve{q}_1\cdots \breve{q}_n\cdots \breve{q}_1)\] which has the effect of reversing the order of products. That is, \[\varsigma(\dprod{B_{ij}}{B_{k\ell}} )= \omega \dprod{B_{k\ell}}{B_{ij}}\] for some invertible $\omega\in U'(\mathfrak{h})$. 
\end{remark}

\subsection{Ordering Relations of \texorpdfstring{$\mathcal{D}(\sp_{2n})$}{D(sp2n)}}
Similar to Section \ref{sec: structure of Bn}, $\mathcal{D}(\sp_{2n})$ is generated as a $U'(\mathfrak{h})$-module by $\mathcal{X} = \set{A_{ij},B_{ij},C_{ij}}$. We impose an order on the generators compatible with the partial order on the weights. \begin{itemize}
    \item For $i\leq j$ and $k\leq \ell$ , $B_{ij} < B_{k\ell}$ when $ij > k\ell$; this is the reverse of the order used in Section \ref{sec: structure of Bn}.
    \item For $i < j$ and $k < \ell$, $A_{ij} < A_{k\ell}$ when $j-i < \ell-k$, or $j-i =\ell-k$ and $i < k$.
    \item For all $i < j$ and $k \leq \ell$, $A_{ij} < B_{k\ell}$
    \item For all $\gamma\in \Delta_+$, $\epsilon(e_\gamma) < e_\gamma$
    \item For all $\gamma, \sigma\in \Delta_+$, if $e_{\gamma} < e_{\sigma}$, then $\epsilon(e_\sigma) < \epsilon(e_\gamma)$.
    \item For all $i<j$, $A_{ii} < A_{jj}$.
\end{itemize}

Note that the Chevalley anti-involution reverses this order for any generators $x,y$ with nonzero weight: if $x < y$, then $\epsilon(y) < \epsilon(x)$. Because of this, $\epsilon$ respects ordered (and misordered) diamond products between elements of nonzero weight; that is, if $\dprod{x}{y}$ is (mis)ordered, then $\epsilon(\dprod{x}{y}) = \bbar{\epsilon(y)}\diamond \bbar{\epsilon(x)}$ is (mis)ordered.

The ordering relations for $\mathcal{B}_n$ given in Theorem \ref{thm: All relations for Bn} can then be used to produce some ordering relations of $\mathcal{D}(\sp_{2n})$.
\begin{theorem}\label{thm: ord rels for non roots}
Every ordering relation of $\mathcal{D}(\sp_{2n})$ whose weight is neither zero nor a root of $\sp_{2n}$ can be obtained by successively applying Zhelobenko automorphisms and/or the Chevalley anti-involution to an initial relation of Theorem \ref{thm: All relations for Bn}.
\end{theorem}
\begin{proof}
If $x,y\in \mathcal{X}$ (are generators of $\mathcal{D}(\sp_{2n})$) and $h\in U'(\mathfrak{h})$ is invertible, then $\breve{q}_i(h \dprod{x}{y})$ has the form $h' \bbar{\tau_i(x)}\diamond\bbar{\tau_i(y)}$ for some invertible $h'\in U'(\mathfrak{h})$ by Lemma \ref{lem: zhel auto for Dsp2n}. Hence when applying a Zhelobenko automorphism to an ordering relation, we can always multiply by an invertible element of $U'(\mathfrak{h})$ so the left-hand side is a product of generators. \\
The Zhelobenko automorphisms act by signed permutations on weights, with $\breve{q}_i$ interchanging $\varepsilon_i$ and $\varepsilon_{i+1}$ for $i < n$, and $\breve{q}_n$ sending $\varepsilon_n$ to $-\varepsilon_n$. For each case we provide a signed permutation that corresponds to a composition of Zhelobenko automorphisms that transforms an initial relation to the target relation, while respecting the order of each diamond product. These are expressed in the abbreviated form $(i_1,i_2,i_3,\cdots i_m)$, meaning $j\mapsto i_j$ for $1\leq j\leq m$, and $j\mapsto j$ for $j > m$.

The following table addresses the most relevant cases. The 'Relation' and 'Initial Relation' columns give the misordered diamond product in the corresponding relation (for 'Initial Relation', we refer to the order used for $\mathcal{B}_n$). The indices $i,j,k,\ell$ are distinct.

\begin{table}[ht] 
\begin{center}
\begin{tabular}{|c|c|c|c|c|}
\hline
Weight & Relation & Initial Relation & Signed Permutation & Conditions \\
\hline
$3\varepsilon_i - \varepsilon_j$ & $\dprod{B_{ii}}{A_{ij}}$ & $\dprod{B_{22}}{B_{12}}$ & $(-j,i)$ &  \\
\hline
$2\varepsilon_i - 2\varepsilon_j$ & $\dprod{B_{ii}}{C_{jj}}$ & $\dprod{B_{22}}{B_{11}}$ & $(-j,i)$ & \\
\hline
$2\varepsilon_i + \varepsilon_j - \varepsilon_k$ & $\dprod{B_{ii}}{A_{jk}}$ & $\dprod{B_{22}}{B_{13}}$ & $(-k,i,j)$ & \\
\hline
$2\varepsilon_i + \varepsilon_j - \varepsilon_k$ & $\dprod{B_{ij}}{A_{ik}}$ & $\dprod{B_{23}}{B_{12}}$ & $(-k,i,j)$ & \\
\hline
$2\varepsilon_i - \varepsilon_j - \varepsilon_k$ & $\dprod{A_{ik}}{A_{ij}}$ & $\dprod{B_{23}}{B_{12}}$ & $(-j,i,-k)$ & $j<k$\\
\hline
$2\varepsilon_i - \varepsilon_j - \varepsilon_k$ & $\dprod{B_{ii}}{C_{jk}}$ & $\dprod{B_{22}}{B_{13}}$ & $(-j,i,-k)$ & $j<k$\\
\hline
$\varepsilon_i + \varepsilon_j + \varepsilon_k - \varepsilon_\ell$ & $\dprod{B_{jk}}{A_{i\ell}}$ & $\dprod{B_{34}}{B_{12}}$ & $(-\ell,i,j,k)$ & $i < j < k$ \\
\hline
$\varepsilon_i + \varepsilon_j + \varepsilon_k - \varepsilon_\ell$ & $\dprod{B_{ik}}{A_{j\ell}}$ & $\dprod{B_{24}}{B_{13}}$ & $(-\ell,i,j,k)$ & $i < j < k$ \\
\hline
$\varepsilon_i + \varepsilon_j + \varepsilon_k - \varepsilon_\ell$ & $\dprod{B_{ij}}{A_{k\ell}}$ & $\dprod{B_{23}}{B_{14}}$ & $(-\ell,i,j,k)$ & $i < j < k$ \\
\hline
$\varepsilon_i + \varepsilon_j - \varepsilon_k - \varepsilon_\ell$ & $\dprod{B_{ij}}{C_{k\ell}}$ & $\dprod{B_{23}}{B_{14}}$ & $(-k,i,j,-\ell)$ & $A_{ik}<A_{j\ell}, A_{jk} < A_{i\ell}$ \\
\hline
$\varepsilon_i + \varepsilon_j - \varepsilon_k - \varepsilon_\ell$ & $\dprod{A_{j\ell}}{A_{ik}}$ & $\dprod{B_{34}}{B_{12}}$ & $(-k,i,j,-\ell)$ & $A_{ik}<A_{j\ell}, A_{jk} < A_{i\ell}$ \\
\hline
$\varepsilon_i + \varepsilon_j - \varepsilon_k - \varepsilon_\ell$ & $\dprod{A_{i\ell}}{A_{jk}}$ & $\dprod{B_{24}}{B_{13}}$ & $(-k,i,j,-\ell)$ & $A_{ik}<A_{j\ell}, A_{jk} < A_{i\ell}$ \\
\hline
$\varepsilon_i + \varepsilon_j - \varepsilon_k - \varepsilon_\ell$ & $\dprod{B_{ij}}{C_{k\ell}}$ & $\dprod{B_{23}}{B_{14}}$ & $(-\ell,i,j,-k)$ & $A_{j\ell}<A_{ik}, A_{i\ell} < A_{jk}$ \\
\hline
$\varepsilon_i + \varepsilon_j - \varepsilon_k - \varepsilon_\ell$ & $\dprod{A_{ik}}{A_{j\ell}}$ & $\dprod{B_{24}}{B_{13}}$ & $(-\ell,i,j,-k)$ & $A_{j\ell}<A_{ik}, A_{i\ell} < A_{jk}$ \\
\hline
$\varepsilon_i + \varepsilon_j - \varepsilon_k - \varepsilon_\ell$ & $\dprod{A_{jk}}{A_{i\ell}}$ & $\dprod{B_{34}}{B_{12}}$ & $(-\ell,i,j,-k)$ & $A_{j\ell}<A_{ik}, A_{i\ell} < A_{jk}$ \\
\hline
\end{tabular} 
\end{center}
\caption{Extending ordering relations of \texorpdfstring{$\mathcal{B}_n$}{Bn} to \texorpdfstring{$\mathcal{D}(\sp_{2n})$}{Dsp2n}}
\label{tab: more relations}
\end{table}

\phantom{.}
This accounts for half of the relations whose weight is not a root of $\sp_{2n}$. The other half are obtained by applying the Chevalley anti-involution $\epsilon$ to all of the relations of Theorem \ref{thm: All relations for Bn} (after applying the order-reversing map $\varsigma$ from Remark \ref{rem: Order reverse}) and Table \ref{tab: more relations}.

\end{proof}

\begin{example}
We claim the ordering relation for $\dprod{B_{12}}{A_{13}}$ can be obtained from the initial relation for $\dprod{B_{23}}{B_{11}}$ (type \ref{rel: 4b}) via the signed permutation $(-3,1,2)$ (which in this notation means $1\mapsto -3, 2\mapsto 1,3\mapsto 2, j\mapsto j$ for $j>3$). More precisely, this signed permutation is achieved by \begin{equation*}
\breve{q}_3\cdots \breve{q}_n\cdots \breve{q}_3\breve{q}_2\breve{q}_1.
\end{equation*}
The initial relations of type \ref{rel: 4b} has the form \begin{equation*}
\dprod{B_{23}}{B_{12}} = c_1 \dprod{B_{12}}{B_{23}} + c_2 \dprod{B_{13}}{B_{22}}
\end{equation*} for some nonzero $c_1,c_2\in U'(\mathfrak{h})$. Applying the signed permutation gives \begin{equation*}
d_1 \dprod{B_{12}}{A_{13}} = d_2 \dprod{A_{13}}{B_{12}} + d_3\dprod{A_{23}}{B_{11}}
\end{equation*} where $d_1,d_2,d_3\in U'(\mathfrak{h})$ are nonzero and $d_1$ is invertible. The most important detail is that the order of each diamond product pair is preserved by the chosen signed permutation, as it is here.

The relations not mentioned in the table are obtained by Chevalley anti-involution. For instance, the ordering relation for $\dprod{A_{32}}{C_{12}}$ is obtained from the ordering relation for $\dprod{B_{12}}{A_{23}}$, and can be expressed as \begin{equation*}
 \dprod{A_{31}}{C_{12}} d_1=  \dprod{C_{12}}{A_{31}} d_2 + \dprod{C_{11}}{A_{32}} d_3
\end{equation*}
\end{example}

\bibliographystyle{abbrv}

\begin{thebibliography}{AST1973}
\bibitem[AST1973]{Ast1973}R. M. Asherova, Yu. F. Smirnov and V. N. Tolstoy, Projection operators for simple Lie groups (Russian), Teoret. Mat. Fiz, vol. 8, pp. 813--825 (1973).
\bibitem[HW2022]{Hart2022} J. T. Hartwig and D. A. Williams II, Diagonal Reduction Algebra for $\mathfrak{osp}(1|2)$, Theoretical and Mathematic Physics, vol. 210, pp. 155--171 (2022).
\bibitem[H2017]{Herl2017} B. Herlemont, Differential calculus on $h$-deformed spaces, SIGMA, vol. 13, pp. 82--109 (2017).
\bibitem[VdH1975]{Homb1975} A. P. Van den Hombergh, A note on Mickelsson's step algebra, Indagationes Mathematicae (Proceedings), vol. 78, No. 1, pp. 42--47. North-Holland. (1975, January).
\bibitem[VdH1976]{Homb1976} A. P. Van den Hombergh, Harish-Chandra Modules and Representations of Step Algebras, Ph. D. Thesis, Katolic University of Nijmegen (1976).
\bibitem[KhO2008]{Kho2008} S. Khoroshkin and O. Ogievetsky, Mickelsson algebras and Zhelobenko operators, J. Algebra, vol. 319, pp. 2113--2165 (2008).
\bibitem[KhO2010]{Kho2010} S. Khoroshkin and O. Ogievetsky, Diagonal Reduction Algebras of $\mathfrak{gl}$ Type, Funct Anal Its Appl 44, 182--198 (2010).
\bibitem[KhO2011]{Kho2011} S. Khoroshkin and O. Ogievetsky, Structure Constants of Diagonal Reduction Algebras of $\mathfrak{gl}$ Type, SIGMA, vol. 7, pp. 64--97 (2011).
\bibitem[KhO2017]{Kho2017} S. Khoroshkin and O. Ogievetsky, Diagonal Reduction Algebra and the Reflection Equation, Israel Journal of Mathematics, vol. 221, pp. 705--729 (2017).
\bibitem[M1973]{Mick1973} J. Mickelsson, Step Algebra of Semi-Simple Subalgebras of Lie Algebras, Reports on Mathematical Physics, vol. 4, pp. 307--218 (1973).
\bibitem[MA2015]{Mud2015} A. Mudrov and T. Ashton, R-matrix and Mickelsson algebras for orthogymplectic quantum groups, Journal of Mathematical Physics, vol. 56 (2015).
\bibitem[T2005]{Tol2005} V. N. Tolstoy, Fortieth Anniversary of Extremal Projector Method for Lie Symmetries, in: Noncommutative Geometry and Representation Theory in Mathematical Physics, pp.371--384 (2005).
\bibitem[T2010]{Tol2010} V. N. Tolstoy, Extremal Projectors for Contragredient Lie (Super) Symmetries (Short Review), Physics of Atomic Nuclei, vol. 74, pp. 1747--1757 (2010).
\bibitem[Zh1987]{Zhe1987} D. P. Zhelobenko, Extremal cocycles of Weyl groups, Functional Analysis and Its Applications, vol. 21, Issue 3,  pp. 183--192 (1987).
\bibitem[Zh1989]{Zhe1989} D. P. Zhelobenko, Extremal projectors and generalized Mickelsson algebras over reductive Lie algebras, Mathematics of the USSR-Izvestiya, vol. 33, pp.  85--100 (1989).
\end{thebibliography}

\end{document}